\documentclass[12pt]{article}
\usepackage[T2A]{fontenc}
\usepackage[cp1251]{inputenc}
\usepackage[russian,english]{babel}
\usepackage{amsmath,amssymb,amsfonts,amsthm}

\newtheorem{theorem}{Theorem}

\newtheorem{corollary}{Corollary}
\newtheorem{lemma}{Lemma}

\newtheorem{claim}{Claim}
\newtheorem*{observation}{Observation}

\newcommand{\ff}[2]{{#1}^{\underline{#2}}}
\let\geq\geqslant
\let\leq\leqslant

\def\P{\mathcal{P}}
\def\Y{\mathcal{Y}}
\def\SY{\mathcal{SY}}
\def\M{\mathcal{M}}
\DeclareMathOperator\SYM{Sym}
\begin{document}
\title{Polynomial Approach to Explicit Formulae for Generalized Binomial Coefficients}
\author{F. Petrov}
\maketitle

\let\thefootnote\relax\footnote{
St. Petersburg Department of
V.~A.~Steklov Institute of Mathematics of
the Russian Academy of Sciences, St. Petersburg State University.
E-mail: fedyapetrov@gmail.com.
Supported by Russian Scientific Foundation grant 14-11-00581.}

\begin{abstract}
We extend the polynomial approach to hook length formula proposed in \cite{KNPV} to
several other problems of the same type, including number of paths formula in the
Young graph of strict partitions. 
\end{abstract}

Multivariate polynomial interpolation, or, on other language, explicit form of Alon's Combinatorial
Nullstellensatz \cite{L}, \cite{KP}, recently proved to be powerful in proving polynomial identities: in \cite{KP} it was
used for the direct short proof of Dyson's conjecture, later generalized in \cite{KN} for
the equally short proof of the $q$-version of Dyson's conjecture, then after additional
combinatorial work it allowed to prove identities of Morris, Aomoto, Forrester (the last was open),
their common generalizations, both in classical and $q$-versions.

Here we outline how this method works in classical theory of symmetric
functions,  ``self-proving'' polynomial identities corresponding to
counting paths in Young graph and ``strict Young graph'', or, in other words,
counting dimensions
of linear and projective representations of the symmetric groups.

We start with a general framework of paths in graded graphs.

\section{Graded graphs}

Let $G$ be a $\mathbb{Z}$-graded countable directed graph with the vertex set 
$V(G)=\sqcup_{i=-\infty}^{\infty} V_i$,
directed edges $(u,v)\in E(G)$ join vertices of consecutive levels: $u\in V_i$, $v\in V_{i+1}$
for some integer $i$.
In what follows  both indegrees and outdegrees of all vertices are bounded. This implies
that the number $P_G(v,u)$ of directed paths from $v$ to $u$ is finite
for any two fixed vertices $v,u$. This number is known as
\emph{generalized binomial coefficient} (usual
binomial coefficient appears for the Pascal triangle). 

Fix a positive integer $k$. Our examples are induced subgraphs of the lattice
$\mathbb{Z}^k$. Vertices are graded by the sum of coordinates,
edges correspond to increasing of any coordinate by 1. That is, indegree 
of any vertex equals its outdegree and equals $k$. It is 
convenient to
identify the vertex set $V$ with monomials $x_1^{c_1}\dots x_k^{c_k}$,
$(c_1,\dots,c_k)\in \mathbb{Z}^k$. Then any edge corresponds to a multiplying
of monomial  by some variable $x_i$. Finite linear combinations of elements of $V$
(with, say, rational coefficients, it is not
essential hereafter) form a ring $\mathbb{Q}[[x_1,\dots,x_k]]$
of Laurent polynomials in the variables $x_1,\dots,x_k$. Not necessary finite
linear combinations form a $\mathbb{Q}[[x_1,\dots,x_k]]$-module which we
denote $\Phi$. For monomial $v\in V$ and $\varphi\in \Phi$ we denote by $[v]\varphi$
a coefficient in $v$ of the series $\varphi$.

Define \textit{minimum} $\min(u,v)$ of two monomials
$u=\prod x_i^{a_i}$, $v=\prod x_i^{b_i}$, as $\min(u,w):=\prod x_i^{\min(a_i,b_i)}$.
If $\min(u,v)=u$ we say that monomial $v$ \textit{majorates} $u$.

Now we describe three examples of graded graphs for which 
there are known explicit formulae for generalized binomial coefficients.

1) Multidimensional Pascal graph $\P_k$. This is a subgraph of
$\mathbb{Z}^k$ formed by the vectors with integral non-negative
coordinates (or, in monomial language, by monomials having all variables
in non-negative power.)  For $k=2$ this graph is isomorphic to Pascal triangle.

2) Restricted Young graph $\Y_k$. This is a subgraph of Pascal graph
formed by vectors $(c_1,\dots,c_k)$ with strictly
increasing non-negative coordinates, $0\leq c_1<c_2<\dots<c_k$.
The vertices of $\Y_k$ may be identified with Young diagrams having at most $k$
rows (to any vertex $(c_1,\dots,c_k)$ of $\Y_k$ we assign a Young diagram with
rows lengths $c_1\leq c_2-1\leq\dots\leq c_k-(k-1)$). Edges correspond
to addition of cells. The usual Young graph has all Young diagrams
as vertices, but when we count number of paths between two diagrams we may
always restrict ourselves to $\Y_k$ with large enough $k$.

3) Graph of strict partitions $\SY_k$. This is a subgraph of Pascal graph
formed by vectors $(c_1,\dots,c_k)$ with non-strictly increasing
non-negative coordinates $0\leq c_1\leq c_2\leq \dots \leq c_k$ satisfying
also the following condition: if $c_i=c_{j}$ for $i\ne j$, then $c_i=0$. 
To any vertex $(c_1,\dots,c_k)$ we may assign a Young diagram
with rows lengths $c_1,\dots,c_k$. So, this diagram has at most $k$
(non-empty) rows and they have distinct lengths. 

The following general straightforward fact connects generalized binomial coefficients
for subgraphs of $\mathbb{Z}^k$ and coefficients of polynomials or power series.

\begin{theorem}\label{coefspaths}
Let $G$ be a subgraph of $\mathbb{Z}^k$, monomials $u,v\in V$ be two vertices of $G$,
$\deg u\geq \deg v$.
Assume that $\varphi\in \Phi$ be a series satisfying the following conditions:

1) $[v]\varphi=1$;

2) if $v'\in V(G)$ and $\deg v=\deg v'$, then $[v']\varphi=0$;

3) if $w\notin V(G)$, but $x_iw\in V(G)$ for some variable $x_i$, then 
$[w](\varphi(x_1+\dots+x_k)^{\deg w-\deg u})=0$.

Then the number of paths from $v$ to $u$ equals 
$$
P_G(v,u)=[u]\left(\varphi(x_1+\dots+x_k)^{\deg u-\deg v}\right).
$$
\end{theorem}

\begin{proof} Induction in $\deg u$. Base $\deg u=\deg v$ follows from 1) and 2). 
Denote $m=\deg u-\deg v$ and assume that the statement is proved for all vertices
of degree less then $m+\deg v$. Let $(u_i,u),i=1,\dots,s$, be all edges of the graph 
$G$ coming to $u$. Clearly $s\leq k$ and without loss of generality $u=x_iu_i$ for $i=1,\dots,s$.
Also note that for $s<i\leq k$ we have $[ux_i^{-1}](\varphi(x_1+\dots+x_k)^{m-1})=0$
by property 3).
Thus
\begin{multline*}
[u]\left(\varphi(x_1+\dots+x_k)^m\right)=\sum_{i=1}^k[ux_i^{-1}]\left(\varphi(x_1+\dots+x_k)^{m-1}\right)=\\
=\sum_{i=1}^s[u_i]\left(\varphi(x_1+\dots+x_k)^{m-1}\right)=\sum_{i=1}^s P_G(v,u_i)=P_G(v,u)
\end{multline*}
as desired.
\end{proof}

Theorem \ref{coefspaths} leads to a natural question: for which labeled graphs $(G,u)$ there exists such a function $\varphi$
that conditions 1), 2), 3) are satisfied? I do not know a full answer. The following statement is at least general enough to cover
all examples of this paper.

\begin{theorem}\label{general_existence} Assume that the $V\subset \mathbb{Z}^k$, $G$ is induced subgraph on $G$ with vertex set $V=V(G)$, and it
satisfies two following conditions:

1) (minimum-closed set) if $u,w\in V$, then also $\min(u,w)\in V$.  

2) (coordinate convexity) if $u,w\in V$, $w=ux_i^m$ for some index $i$ and positive integer $m$, then also $ux_i^s\in V$ for all
$0\leq s\leq m$. 

Then for any $v\in U$ there exists a function $\varphi$ satisfying conditions of Theorem \ref{coefspaths}. 
\end{theorem}

\begin{proof}
Call a monomial $u\in \mathbb{Z}^k$ \textit{special} if either $u\in V$, $\deg u=\deg v$ (in particular, $v$ itself is special), or 
$\deg u\geq \deg v$,  $u\notin V$, $x_iu\in V$ for some variable $x_i$. It suffices to prove that there exists function
$\varphi$ homogeneous of degree $\deg v$ with any prescribed values of $[u](\sum x_i)^{\deg u-\deg v} \varphi$
for all special $u$.  This is a linear system on coefficients of $\varphi$. 

By replacing $V$ to $v^{-1}V$ we may suppose that $v=1$.
For any special monomial $u$ define its son $S(u)$: if $\deg u=0$, then $S(u)=0$, if $u\notin V$,
$x_iu\in V$, then $S(u)=x_i^{-\deg u} u$ (if different $u$ satisfy $x_iu\in V$, choose any). Key claim is as follows:

if $u\ne w$ are two special monomials and $w$ majorates $S(u)$, then 
$\deg w>\deg u$. 

Proof of the key claim. Assume that on the contrary $\deg w\leq \deg u$. At first, if $\deg u=0$, then $\deg w=0$ and
$w=S(u)=u$, a contradiction. Thus $d:=\deg u>0$, we may suppose that $x_1u\in V$,  $S(u)=x_1^{-d} u$. Denote
$w=x_1^{a_1-d}x_2^{a_2}\dots x_k^{a_k}u$, $a_i\geq 0$, $d\geq a_1+\dots+a_k$, and either $w\in V$ or $x_iw\in V$
for some $i$. Since  $w\ne u$, we have $d<a_1$. 
Denote $u_0:=\min(x_iw,x_1u)$, if $x_iw\in V$ and $u_0=\min(w,x_1u)$ if $w\in V$. Then 
$u_0=x_1^{a_1-d+\varepsilon}u$, where $\varepsilon\in \{0,1\}$ and $u_0\in V$ since $V$ is minimum-closed set. 
Coordinate convexity implies 
that if both $u_0,x_1u$ belong to $V$,
so does $u$. A contradiction. 

Now we start to solve our linear system for  coefficients of $\varphi$. For any special monomial $u$ we have a linear relation on 
the monomials of degree 0 majorated by $u$. Between them there is a monomial $S(u)$, and it does not appear in
relations corresponding to special monomials $w\ne 0$ with $\deg w\leq \deg u$. It allows to fix coefficients of $\varphi$
in appropriate order (by increasing the degree of $u$) and satisfy all our relations.
\end{proof}

\section{Observation on polynomials}

For verifying polynomial identities which allow
to calculate  coefficients  we use the following lemma in the spirit of Combinatorial 
Nullstellensatz \cite{Alon} of N.~Alon.

Let $A_i=\{a_{i0},a_{i1},\dots,a_{in}\}\subset K$,
where $K$ is a field, $i=1,2,\dots,k$. 

\begin{observation}\label{polyidentity} 
A polynomial $f(x_1,\dots,x_k)\in K[x_1,\dots,x_k]$
of degree at most $n$ is uniquely determined by its values on the combinatorial simplex
$$\Delta=\left\{A(t_1,\dots,t_k):=(a_{1t_1},a_{2t_2},\dots,a_{kt_k}),\sum t_i\leq n\right\}.$$
\end{observation}

\begin{proof}
The number of points in $\Delta$ exactly equals to the dimension of the space
of polynomials with degree at most $n$ in
$k$ variables. Hence it suffices to check either existence or
uniqueness of the polynomial with degree at most $n$ and prescribed values on $\Delta$.

Both tasks are easy as we may see:

1) Existence. Induction in $n$. Base $n=0$ is clear. Assme that $n>1$
and for the simplex $\Delta'$, which corresponds to inequality $\sum t_i\leq n-1$,
there exists a polynomial $f(x_1,\dots,x_k)$ with degree at most $n-1$ and prescribed
values on $\Delta'$. For any point 
$A(t_1,\dots,t_k)$ with $\sum t_i=n$ we have a polynomial
$$\prod_{1\leq i\leq k,0\leq s_i<t_i} (x_i-a_{is_i}).$$
vanishing on all points of $\Delta$ but $A(t_1,\dots,t_k)$.
Appropriate linear combination of $f(x)$ and such polynomials
gives a polynomial with prescribed values on $\Delta$ (and degree at most $n$).

2) Uniqueness. It suffices to prove that polynomial $f(x_1,\dots,x_k)$
with degree at most $n$, which vanishes on $\Delta$, identically equals 0.
Assume the contrary. Let $x_1^{t_1}\dots x_k^{t_k}$ be a highest degree term in
$f$. The set $\Delta$ contains the product $B_1\times \dots \times B_k$,
where $B_i=\{a_{i0},\dots,a_{it_i}\}$, $|B_i|=t_i+1$. By Combinatorial Nullstellensatz
$f$ can not vanish on
$\prod B_i$. A contradiction.
\end{proof}

For $A_i=\{0,1,\dots,n\}$ we get  \emph{the standard simplex} 
$$\Delta_k^n=\{(t_1,\dots,t_k):t_i\geq 0, t_i\in \mathbb{Z}, \sum t_i\leq n\}.$$
It is a main partial case for us. In the theory of $q$-identities its $q$-analogue, which corresponds
to the sets $\{1,q,\dots,q^{n}\}$, plays analogous role.

We use notations $\ff{x}{n}=x(x-1)\dots(x-n+1)$, $\binom{x}{n}=\ff{x}{n}/n!$.

The following particular case of interpolation on $\Delta_k^n$ appears to be useful.

\begin{lemma}\label{interpolation}
If $f\in K[x_1,\dots,x_k]$, $\deg f\leq n$ and $f$ vanishes on $\Delta_k^{n-1}$, then
$$
f(x_1,\dots,x_k)=\sum_{c_1+\dots+c_k=n} f(c_1,\dots,c_k)\cdot \prod_{i=1}^k \binom{x_i}{c_i}=
\sum_{c_1+\dots+c_k=n} \frac{f(c_1,\dots,c_k)}{c_1!\dots c_k!}\cdot \prod_{i=1}^k \ff{x_i}{c_i}.
$$
\end{lemma}

\begin{proof} It suffices to check the equality of values on $\Delta_k^n$.
Both parts vanish on $\Delta_k^{n-1}$. If $(x_1,\dots,x_k)$ is a point on $\Delta_k^n\setminus \Delta_k^{n-1}$,
i.e. $\sum x_i=n$, then all summands on the right vanish except (possibly) the summand with $c_i=x_i$ for all $i$,
and its value just equals to $f(c_1,\dots,c_k)=f(x_1,\dots,x_k)$, i.e. the value of LHS in the same point, as desired.
 \end{proof}

\section{Chu--Vandermonde identity and multinomial coefficient}

This paragraph contains very well known results, but it shows how the scheme works.

We start our series of applications of Observation with the multinomial version
of  Chu--Vandermonde identity:

\begin{equation}\label{vandermonde}
\binom{x_1+\dots+x_k}{n}=\sum_{c_1+\dots+c_k=n} \prod_{i=1}^k \binom{x_i}{c_i}.
\end{equation}

This immediately follows from Lemma \ref{interpolation}.

The identity \eqref{vandermonde} has a following form with falling factorials:

\begin{equation}\label{vandermonde1}
\ff{(x_1+\dots+x_k)}{n}=\sum_{m_1+\dots+m_k=n} \frac{n!}{\prod m_i!} \prod_{i=1}^k \ff{x_i}{m_i}.
\end{equation}

Taking only leading terms in both sides we get Multinomial Theorem

\begin{equation}\label{multinomial}
(x_1+\dots+x_k)^n=\sum_{m_1+\dots+m_k=n} \frac{n!}{\prod m_i!} \prod_{i=1}^k {x_i}^{m_i}.
\end{equation}

Returning back to graded graphs, for multidimensional Pascal Graph $\P_k$ we get the formula
for the number of paths from origin 1 to any vertex $v=\prod x_i^{m_i}$, $m_i\geq 0$, $n:=\sum m_i=\deg v$:
$$
P_{\P_k} (1,v)=[v]\left(\sum x_i\right)^{n} = \frac{n!}{\prod m_i!}.
$$
This immediately follows from Theorem \ref{coefspaths} for $\varphi=1$ and identity \eqref{multinomial}.

\section{Hook length formula}

We start with polynomial identity which is in a sense similar to Chu--Vandermonde identity \eqref{vandermonde1}:

\begin{theorem}\label{identity_hooks}
\begin{align*}
\prod_{i<j} (x_j-x_i) \ff{\left(\sum_{i=1}^k x_i-k(k-1)/2\right)}{n} =\\
\sum_{n_1+\dots+n_k=n+k(k-1)/2} 
\frac{n!}{\prod n_i!} \prod_{i<j} (n_j-n_i)\cdot \prod \ff{x_i}{n_i}.
\end{align*}
\end{theorem}

\begin{proof}
By Lemma \ref{interpolation} it suffices to check that LHS vanishes on
$\Delta_k^{n+k(k-1)/2-1}$. 
Let $x_i$'s be non-negative integers and $\sum x_i<n+k(k-1)/2$.
If $\sum x_i<k(k-1)/2$ then some multiple $x_i-x_j$
vanishes, otherwise $y:=\sum x_i-k(k-1)/2$ is non-negative and $y<n$, hence $\ff{y}{n}=0$. 
\end{proof}

As in Vandermonde--Chu case, we may take only leading terms and get an identity
\begin{equation}\label{hooksleading}
\prod_{i<j} (x_j-x_i)\left(x_1+\dots+x_k\right)^n=\sum_{n_1+\dots+n_k=n+k(k-1)/2} 
\frac{n!}{\prod n_i!} \prod_{i<j} (n_j-n_i)\cdot \prod {x_i}^{n_i}.
\end{equation}

\begin{corollary}\label{hooks_graph} For any vertex $v=x_1^{n_1}\dots x_k^{n_k}$, $\sum n_i=n+k(k-1)/2$ of Young graph $\Y_k$ the number 
$P_{\Y_k}(v_0,v)$ to this vertex from $v_0=(0,1,\dots,k-1)$ equals
\begin{equation}\label{hooks}
P_{\Y_k}(v_0,v)=\frac{n!}{\prod n_i!} \prod_{i<j} (n_j-n_i)=[v] \prod_{i<j}(x_j-x_i) \left(\sum x_i\right)^n.
\end{equation}
\end{corollary}

\begin{proof} Take $\varphi=\prod_{i<j}(x_j-x_i)$ and apply Theorem \ref{coefspaths}. 
Conditions 1) and 2) follow from \eqref{hooksleading} with $n=0$ (actually, this is just
the Vandermonde determinant formula). For checking 3) note that such $w=\prod x_i^{n_i}$
satisfies either $n_i<0$ for some $i$ or $n_i=n_j$ for some $i,j$. In both cases \eqref{hooksleading}
yields that corresponding coefficient vanishes. 
\end{proof}

Note that our proof of the formula \eqref{hooks} does not use Multinomial theorem, 
but is proved in the same way and the proof is almost equally short. In the rest part of this section
we explain the relation with hook lenths of Young diagram, this is just for the sake
of completeness.  

Recall that the graph $\Y_k$ may be viewed as the graph of Young diagrams having at most $k$
rows. For a vertex $v\in \Y_k$, $v=x_1^{n_1}\dots x_k^{n_k}$, corresponding diagram $\lambda(v)$
has $k$ (possible empty) rows with lengths $n_1\leq n_2-1\leq \dots\leq n_k-(k-1)$. Edges of
the graph $\Y_k$ correspond to adding cells, and paths correspond to \emph{skew standard Young tableaux}:
for any path with, say, $m$ edges put numbers $1,2,\dots,m$ in corresponding adding cells.
On this language expression \eqref{hooks} counts the number of standard Young tableaux
of the shape $\lambda(v)$. Assuming $n_1>0$ (i.e. the number of rows equals $k$) 
we may interpret parameters $n_1,\dots,n_k$
as \emph{hook lengths} of $k$ cells in the first column. Recall that (now specify that largest column in Young diagram 
is the leftmost and largest row the lowest) a hook of a cell $X$ in Young diagram is a union of $X$; all cells 
in the same column which are higher then $X$; all cells in the same
row which are on the right to $X$. 

\begin{claim}\label{hooks_products} In above notations
the product of hooks lengths of all cells in the Young diagram $\lambda(v)$ equals
$$
\prod h(\square)=\frac{\prod n_i!}{\prod_{i<j} (n_j-n_i)}.
$$ 
\end{claim}

\begin{proof} Assume that cells $a,b$ of a Young diagram lie in the 
same row and $a,c$ in the same column. Let $d$ be such a cell that $abdc$ is a rectangle.
If $d$ belongs to a diagram then $h(a)<h(b)+h(c)$, otherwise $h(a)>h(b)+h(c)$. Hence we always
have \emph{inequality}: $h(a)\ne h(b)+h(c)$. 
Now $n_1,\dots,n_k$ are hooks lengths of cells in the first column.
In $i$-th row there are $n_i-(i-1)$ cells, and their hooks lengths are distinct numbers
from 1 to $n_i$, with $i-1$ values excluded, and those
excluded values are $n_i-n_1$, $n_i-n_2$, $\dots$, $n_i-n_{i-1}$ --- by \emph{inequality}.
It suffices to multiply by $i=1,2,\dots,k$.  
\end{proof}

The above claim allows to formulate Corollary \ref{hooks_graph} in the form of \emph{hook length formula} \cite{FRT}:

\begin{theorem}\label{hooks_formula}[Hook length formula]
The number of standard Young tableaux of a given shape $\lambda$ with $n$ cells 
equals $n!/\prod_{\square} h(\square)$,
where product is taken over all cells of $\lambda$.
\end{theorem}

\section{Skew Young tableaux }

Here we get generalizations of Theorem \ref{identity_hooks}, corresponding to counting
the paths between two arbitrary vertices of $\Y_k$ (i.e. the number of skew Young tableaux of a given shape).

The role of Vandermonde determinant $\prod_{i<j} (x_j-x_i)=\det (x_i^{j-1})_{i,j}$ is played by alternating determinants
$$
a_{m_1,\dots,m_k}(x_1,\dots,x_k)=\det\left(x_i^{m_j}\right);\,b_{m_1,\dots,m_k}(x_1,\dots,x_k)=\det\left(\ff{x_i}{m_j}\right).
$$
The following identity specializes to Theorem \ref{identity_hooks} for $m_i=i-1$, $i=1,2,\dots,k$.

\begin{theorem}\label{identity_hooks} If $m_1<\dots<m_k$ are distinct non-negative integers, $m=\sum m_i$, then
\begin{align*}
b_{m_1,\dots,m_k}(x_1,\dots,x_k) \ff{\left(\sum x_i-m\right)}{n} =\\
\sum_{n_1+\dots+n_k=n+m} 
\frac{n!}{\prod n_i!} b_{m_1,\dots,m_k}(n_1,\dots,n_k)\cdot \prod \ff{x_i}{n_i}.
\end{align*}
\end{theorem}

\begin{proof} Due to Lemma \ref{interpolation}
 it suffices to check that LHS vanishes on $\Delta_k^{n+m-1}$. 
Fix a point $(x_1,\dots,x_k)\in \Delta_k^{n-1}$. Let $y_1\leq \dots \leq y_k$ be 
the increasing permutation of $x_i'$s. If $y_i<m_i$ for some $i$, then a matrix
$(\ff{y_i}{m_j})$ is singular as it has a $i\times (n-i+1)$ minor of zeros, matrix 
$(\ff{x_i}{m_j})$ is therefore singular too. If $y_i\geq m_i$ for all $i$, then denoting
$y:=\sum x_i-m=\sum (y_i-m_i)\geq 0$ we have $0\leq y<n$, hence $\ff{y}{n}=0$.
\end{proof}

Taking leading terms we get the following identity for homogeneous polynomials:
\begin{equation}\label{skew_homogeneous}
a_{m_1,\dots,m_k}(x_1,\dots,x_k) \left(\sum x_i\right)^{n} =\\
\sum_{n_1+\dots+n_k=n+m} 
\frac{n!}{\prod n_i!} b_{m_1,\dots,m_k}(n_1,\dots,n_k)\cdot \prod {x_i}^{n_i}.
\end{equation}

\begin{corollary}\label{hooks_graph} Let $v_1=x_1^{m_1}\dots x_k^{m_k}$, 
$v_2=x_1^{n_1}\dots x_k^{n_k}$ be two vertices of the Young graph $\Y_k$ such that 
$n_i\geq v_i$ for all $i=1,2,\dots,k$. Denote $\sum m_i=m$, $\sum n_i=n+m$. Then the number 
$P_{\Y_k}(v_1,v_2)$ of paths from $v_1$ to $v_2$ equals
\begin{equation}\label{skew}
P_{\Y_k}(v_1,v_2)=\frac{n!}{\prod n_i!} b_{m_1,\dots,m_k}(n_1,\dots,n_k)=[v_2]  a_{m_1,\dots,m_k}(x_1,\dots,x_k) \left(\sum x_i\right)^n.
\end{equation}
\end{corollary}

\begin{proof} Take $\varphi=a_{m_1,\dots,m_k}(x_1,\dots,x_k)$ and apply Theorem \ref{coefspaths}. 
Conditions 1) and 2) follow from \eqref{skew_homogeneous} with $n=0$ (or from expanding the determinant).
For checking 3) note that such $w=\prod x_i^{n_i}$
satisfies either $n_i<0$ for some $i$ or $n_i=n_j$ for some $i,j$. In both cases \eqref{skew_homogeneous}
yields that corresponding coefficient vanishes. 
\end{proof}

The value $b_{m_1,\dots,m_k}(n_1,\dots,n_k)$ has a combinatorial interpretation
following from the  Lindstr\"om -- Gessel -- Viennot lemma: up to a multiple $\prod m_i!$
it is a number of semistandard Young tableaux
of a given shape and content. See details in \cite{F}.

A.~M.~Vershik pointed out that similar results are known for the graph of strict diagrams. It
also may be included in our framework. 

\section{Strict diagrams}

For counting the number of paths in teh graph $\SY_k$ we need series which are not polynomials. 

Let $x_1,\dots,x_k$ be variables (as before). Consider the set $\M$ of rational functions in
hose variables with denominator $\prod_{i<j} (x_i+x_j)$. 
Expand such functions in Laurent series in $x_1,x_2/x_1,x_3/x_2,\dots,x_k/x_{k-1}$ (i.e. $(x_i+x_j)^{-1}=x_i^{-1}-x_jx_i^{-2}+x_j^2x_i^{-3}-\dots$
for $i<j$). Define the value of such a function in a point
$(c_1,\dots,c_k)$ with non-negative coordinates: if coordinates are positive,
just substitute them in a function, if some coordinates vanish, replace them by positive numbers
$t,t^2,\dots$ (in such an order), and let $t$ tend to $+0$. What we actually need is 
that for $i<j$ the value $x_j/(x_i+x_j)$ with $x_i=x_j=0$ equals 0. 
Each function  $f\in \M$ may be expanded as
$f=P[f]+Q[f]$, where $P[f]$ is a polynomial in $x_1,\dots,x_k$, 
and in $Q[f]$ each term $\prod x_i^{c_i}$ contains at least one variable $x_i$ in a negative power
$c_i<0$. We say that $P[f]$ is a \emph{polynomial component} of $f$ and $Q[f]$ is an \emph{antipolynomial component} 
of a function $f(x_1,\dots,x_k)$. 

\begin{lemma}\label{antipoly} Define a function $f_n(x_1,\dots,x_k)\in \M$ by the formula
$$
f_n(x_1,\dots,x_k)=\prod_{1\leq i<j\leq k} \frac{x_i-x_j}{x_i+x_j} \cdot \ff{(x_1+\dots+x_k)}{n}.
$$
Then 

1) its antipolynomial component $Q[f_n]$ vanishes on a standard simplex $\Delta_k^n$.

2) if $c_1,\dots,c_k$ are integers such that $c_j<0$, $c_i\geq 0$ for $i=j+1,\dots,k$, then
$[\prod x_i^{c_i}]f_n=0$.
\end{lemma}

\begin{proof} We may suppose that $k=2d$ is even (else replace $k$ to $k+1$ and put $x_{k+1}=0$). 
Consider the following antisymmetric $k\times k$ matrix: $a_{ij}=(x_i-x_j)/(x_i+x_j)$. 
Note that its Pfaffian lies in $\M$, is is homogeneous of order 0, it vanishes for $x_i=x_j$ 
and is singular for $x_i=-x_j$. Thus up to a constant multiple (it is
not hard to verify that actually up to a sign) it equals $\prod_{1\leq i<j\leq k} \frac{x_i-x_j}{x_i+x_j}$.
Pfaffian is an alternating sum of expressions like
$$\prod_{i=1}^d \frac{\xi_i-\zeta_i}{\xi_i+\zeta_i},$$ 
where $\{\xi_1,\zeta_1,\dots,\xi_d,\zeta_d\}=\{x_1,\dots,x_k\}$,
On the other hand, Vandermonde identity \eqref{vandermonde} allows
to express a falling factorial $\ff{(x_1+\dots+x_k)}{n}$ as a linear expressions like
$$\prod_{i=1}^d \ff{(\xi_i+\zeta_i)}{\alpha_i},\,\alpha_i\geq 0, \alpha_1+\dots+\alpha_d=n.$$
Fixing partition into pairs and after that exponents $\alpha_1,\dots,\alpha_d$ we reduce both claims to the expression
$$
F(x_1,\dots,x_k)=\prod_{i=1}^d \frac{\xi_i+\zeta_i}{\xi_i+\zeta_i}\cdot \ff{(\xi_i+\zeta_i)}{\alpha_i}.
$$
Note that variables are separated here, thus the polynomial component of this product
is just the product
of polynomial components of the multiples. We have
$$
Q[F]=F-P[F]=\prod_{i=1}^d \frac{\xi_i+\zeta_i}{\xi_i+\zeta_i}\cdot \ff{(\xi_i+\zeta_i)}{\alpha_i}-\prod_{i=1}^d P\left[\frac{\xi_i+\zeta_i}{\xi_i+\zeta_i}\cdot 
\ff{(\xi_i+\zeta_i)}{\alpha_i}\right].
$$
If $\alpha_i\geq 1$, the corresponding multiple is a polynomial. If  $\alpha_i=0$, we have $\xi_i=x_a,\zeta_i=x_b$ for some indexes $a<b$,
we use
a relation $P[(x_a-x_b)/(x_a+x_b)]=1$. Substituting $(x_1,\dots,x_k)\in \Delta_k^n$ we see that if
$\xi_i+\zeta_i<\alpha_i$ for some index $i$, both minuend and the subtrahend take zero value, otherwise
$\xi_i=\zeta_i=0$ for all $i$ with $\alpha_i=0$, thus corresponding multiples in 
minuend and the subtrahend are equal.

For proving 2) note that we should have $x_j=\xi_s$ for some $s$, but if $\xi_s$ is taken in negative power, then $\zeta_s=x_{b_s}$
must be taken in positive power and $b_s>j$.  
\end{proof}

Modulo this lemma everything is less or more the same as in previous sections.

\begin{theorem}\label{polycomponent} Polynomial component of the function $f_n(x_1,\dots,x_k)$ equals
$$
P\left[\prod_{1\leq i<j\leq k} \frac{x_i-x_j}{x_i+x_j} \cdot \ff{(x_1+\dots+x_k)}{n}\right]=\sum_{m_1+\dots m_k=n} \prod_{i<j} \frac{m_i-m_j}{m_i+m_j}\cdot \frac{n!}{\prod m_i!}\cdot \prod \ff{x_i}{m_i}
$$
(recall that $\frac{m_i-m_j}{m_i+m_j}$ for $m_i=m_j=0$ is equal to 1). 
\end{theorem}

\begin{proof}
Both parts are polynomials of degree at most $n$, thus it suffices to check that their values in each point
$(c_1,\dots,c_k)\in \Delta_k^n$ are equal. 
The polynomial component of  $f_n(x_1,\dots,x_k)$ takes the same values on $\Delta_k^n$ as the function $f_n$ itself. 
Thus by Lemma \ref{interpolation} it suffices to check that $f_n$ vanishes on
$\Delta_k^{n-1}$. But already the multiple $\ff{(x_1+\dots+x_k)}{n}$
vanishes on $\Delta_k^{n-1}$.
\end{proof}

\begin{corollary}\label{cor1}
Coefficient $H(m_1,\dots,m_k)$ of the Laurent series $\prod_{1\leq i<j\leq k} \frac{x_i-x_j}{x_i+x_j} \cdot (x_1+\dots+x_k)^n$ in monomial 
$\prod x_i^{m_i}$, $m_i\geq 0$,
equals $\prod_{i<j} \frac{m_i-m_j}{m_i+m_j}\cdot \frac{n!}{\prod m_i!}$.
\end{corollary}

\begin{corollary} If $u=\prod x_i^{m_i}$, is a vertex of $\SY_k$, , $n:=\sum m_i$, then the number of paths from the origin to $v$ equals
$$
P_{\SY_k}(1,u)=\prod_{i<j} \frac{m_i-m_j}{m_i+m_j}\cdot \frac{n!}{\prod m_i!}.
$$
\end{corollary}

\begin{proof} Apply Theorem \ref{coefspaths} to the function $\varphi=\prod_{1\leq i<j\leq k} \frac{x_i-x_j}{x_i+x_j}$.
After that result directly follows from Corollary \ref{cor1}, so it suffices to check all conditions of Theorem \ref{coefspaths}.
Condition 1) follows from Corollary \ref{cor1} for $n=0$ (or just from common sense). In condition 2) there is nothing to check since
1 is the unique vertex of $\SY_k$ with degree 0. Let's check condition 3). Assume that $w=\prod x_i^{c_i}$ is
so that $x_iw$ is a vertex of $\SY_k$ but $w$ is not. There are two cases: either all coordinates of 
$w$ are non-negative, and $c_j=c_l>0$ for some $j,l$, or $c_i=-1$, $c_j= 0$ for all $j\geq i$. In the first case apply
Corollary \ref{cor1}, in the second case apply statement 2) of Lemma \ref{antipoly}. 
\end{proof}

\section{Skew strict Young tableaux}

Here we give an identity corresponding to V.~N.~Ivanov's \cite{I} formula for the number of paths between any to vertices of the graph $\SY_k$,
or, in other words, for the number of strict  skew Young tableaux of a given shape.

Let $v=\prod_{i=1}^k x_i^{m_i}$, $m_1>m_2>\dots>m_{\ell}>m_{\ell+1}=0=\dots=m_k$, be a vertex of $\SY_k$, $u=\prod x_i^{n_i}$
be another vertex of $\SY_k$ and $n_i\geq m_i$ for all $i$ (thus there exist some path from $u$ to $v$). Denote $m=\sum m_i$,
$n=\sum n_i$.

For counting such paths we introduce the following polynomial (Ivanov in \cite{I} attributes them to A. Okounkov).
$$
\psi_{v}(x_1,\dots,x_k)=\frac1{(k-\ell)!}\SYM \left(\prod_{i\leq \ell} \ff{x_i}{m_i} \prod_{i\leq \ell, i<j} \frac{x_i+x_j}{x_i-x_j}\right).
$$
Here $\SYM F(x_1,\dots,x_k)=\sum_{\pi} F(x_{\pi_1},\dots,x_{\pi_k})$, where summation is taken over all $k!$
permutations of numbers $1,\dots,k$. (For concluding that it is indeed a polynomial note that any multiple $x_i-x_j$ in denominator
disappears after natural pairing of summands. It is super-symmetric
polynomial, but we do not use this fact.)

Then define a function
$$
\varphi_v(x_1,\dots,x_k)=\prod_{i<j} \frac{x_i-x_j}{x_i+x_j}\cdot \psi_v(x_1,\dots,x_k).
$$

The following theorem generalizes results of the previous section.

\begin{theorem} Denote $g(x_1,\dots,x_k)=\varphi_v(x_1,\dots,x_k)\cdot \ff{(x_1+\dots+x_k-m)}{n-m}$.

1) Antipolynomial component $Q[g]$ vanishes on the simplex $\Delta_k^n$. 

2) Polynomial component of $g$ has an expansion
\begin{equation}\label{Ivanov_identity}
P[g]=\sum_{c_1+\dots+c_k=n,c_i\geq 0}
\frac{(n-m)!}{\prod c_i!} \cdot \varphi_v(c_1,\dots,c_k)\cdot \prod \ff{x_i}{c_i}
\end{equation}

3) Number of paths from $v=\prod x_i^{m_i}$ to $u=\prod x_i^{n_i}$ equals
$$
\frac{(n-m)!}{\prod n_i!} \cdot \varphi_v(n_1,\dots,n_k).
$$
\end{theorem}

\begin{proof}
1) Fix a permutation $\pi$ and prove the statement for the antipolynomial component of the corresponding summand
in the definition of $\varphi_v$.
Denote $y_i=x_{\pi_i}$. Note that up to a sign our summand is 
$$
\ff{(y_1+\dots+y_k-m)}{n-m}\prod \ff{y_i}{m_i} \prod_{\ell<i<j} \frac{y_i-y_j}{y_i+y_j}. 
$$
We expand $\prod_{\ell<i<j} \frac{y_i-y_j}{y_i+y_j}$ as a Pfaffian as explained in the proof of Lemma \ref{antipoly}
(if $k-\ell$ is odd also do the same thing as before: add a vanishing variable; so let $k-\ell=2d$ be even). Expand this Pfaffian, and
take a summand like $\prod_{i=1}^{d} \frac{\xi_i-\zeta_i}{\xi_i+\zeta_i}$, where
$\{\xi_1,\zeta_1,\dots,\xi_d,\zeta_d\}=\{y_{\ell+1},\dots,y_k\}$.

Expand also $\ff{(x_1+\dots+x_k-m)}{n-m}$ by Chu--Vandermonde identity \eqref{vandermonde} as a linear combination of
terms like
$$
\ff{(y_1-m_1)}{\alpha_1}\dots \ff{(y_\ell-m_\ell)}{\alpha_\ell}\ff{(\xi_1+\zeta_1)}{\beta_1}\dots \ff{(\xi_d+\zeta_d)}{\beta_d},\,\, \sum \alpha_i+\sum \beta_i=n-m.
$$

Thus it suffices to prove that antipolynomial component of the following product vanishes on $\Delta_k^n$:
$$
\prod_{i\leq \ell} \ff{y_i}{m_i+\alpha_i} \prod_{i=1}^d \frac{\xi_i-\zeta_i}{\xi_i+\zeta_i} \cdot \ff{(\xi_i+\zeta_i)}{\beta_i}.
$$
Variables are separated, hence the polynomial component $P[\prod\dots]$ of the product 
is a product $\prod P[\dots]$ of polynomial components. IIt suffices to verify that whenever $y_i,\xi_i,\zeta_i$
are non-negative integers with sum at most $n$, the values of $\prod (\dots)$ and $\prod P[\dots]$ are equal. 
If $\beta_i\geq 1$, the corresponding multiple $\frac{\xi_i-\zeta_i}{\xi_i+\zeta_i} \cdot \ff{(\xi_i+\zeta_i)}{\beta_i}$
is a polynomial. If  $\beta_i=0$, we have $\xi_i=x_a,\zeta_i=x_b$ for some indexes $a<b$,
then $P[(x_a-x_b)/(x_a+x_b)]=1$. Substituting $(x_1,\dots,x_k)\in \Delta_k^n$ we see that if
$y_i<m_i+\alpha_i$ or $\xi_i+\zeta_i<\beta_i$ for some index $i$, both products take zero value, otherwise
$\xi_i=\zeta_i=0$ for all $i$ with $\beta_i=0$, thus corresponding multiples (values of function and its polynomial part)
are equal.

Note that as in lemma \ref{antipoly} we may also conclude that if $c_1,\dots,c_k$ are integers such that $c_j<0$, $c_i\geq 0$ for $i=j+1,\dots,k$, then
$[\prod x_i^{c_i}]g=0$.

2) The values of $P[g]$ on $\Delta_k^n$ are the same as values of $g$. 
Any summand of the above expansion for $g$ vanishes on $\Delta_k^{n-1}$. Thus it suffices to use Lemma \ref{interpolation}. 

3) Apply Theorem \ref{coefspaths} to the function $\varphi_v$ (or its leasing part,
it is a matter of taste).
It suffices to check all conditions of Theorem \ref{coefspaths}. 
Conditions 1) and 2) follow from above expansion of $g$ (with $n=0$). There are exactly $(k-\ell)!$ permutations with
$y_i=x_i$, $i=1,\dots,\ell$, for each of them we get coefficient 1 in monomial $v$ and coefficient 0
in other monomials of degree $m$. For other permutations we do not get non-zero coefficients
in monomials which are vertices of $\SY_k$. Let's check condition 3). Assume that $w=\prod x_i^{c_i}$ is
so that $x_iw$ is a vertex of $\SY_k$ but $w$ is not. There are two cases: either all coordinates of 
$w$ are non-negative, and $c_j=c_l>0$ for some $j,l$, or $c_i=-1$, $c_j= 0$ for all $j\geq i$. In the first case apply
identity \eqref{Ivanov_identity}, in the second case apply the above remark after the proof of part 1) of Theorem. 
\end{proof}

This work originated in collaboration with G. K\'arolyi, Z.L. Nagy and V. Volkov. Idea to study the graph
of strict diagrams in the same spirit belongs to A. M. Vershik. To all of them I am really grateful.


\begin{thebibliography}{99}

\bibitem{FRT} Frame, J. S., Robinson, G. de B. and Thrall, R. M. The hook graphs of the symmetric group. Canad. J. Math. 6 (1954), pp. 316–-325.

\bibitem{Alon} N. Alon, 
Combinatorial Nullstellensatz,
Combin. Probab. Comput. 8 (1999), pp.  7--29.

\bibitem{KNPV} G. K\'arolyi, Z.L. Nagy, F. Petrov, V. Volkov. A new approach to constant term identities and Selberg-type integrals. Adv. Math., to appear.

\bibitem{L} M. Laso\'n, 
A generalization of Combinatorial Nullstellensatz,
Electron. J. Combin. 17 (2010) \#N32, 6 pages.

\bibitem{KP} R.N. Karasev, F.V. Petrov, 
Partitions of nonzero elements of a finite field into pairs, 
Israel J. Math. 192 (2012) 143--156.

\bibitem{KN} Gy. K\'arolyi, Z.L. Nagy,
A simple proof of the Zeilberger--Bressoud $q$-Dyson theorem. Proc. AMS 142 (2014), no. 9,  pp. 3007–-3011.

\bibitem{I} V. N. Ivanov. Dimensions of skew-shifted young diagrams and projective characters of the infinite symmetric group. 
Journal of Mathematical Sciences 96 (1999), no. 5, pp. 3517--3530.

\bibitem{F} 
M. Fulmek. Viewing determinants as nonintersecting lattice paths yiekds classical
determinantal identities bijectively. Electron. J. Combin. 19 (2012) \#P21, 46 pages.
\end{thebibliography}
\end{document}